\documentclass[journal,twoside,web]{ieeeconf}
\IEEEoverridecommandlockouts                              
\usepackage{amsmath} 
\usepackage{amssymb}  
\DeclareMathAlphabet{\mathbbold}{U}{bbold}{m}{n}
\usepackage{amsmath,stmaryrd,graphicx,float}
\usepackage{subcaption} 
\usepackage{enumerate}
\usepackage{xcolor}

\definecolor{subsectioncolor}{rgb}{0,0.541,0.855}

\usepackage[backend=bibtex,
style=ieee,
bibencoding=ascii,
]{biblatex}
\usepackage{amssymb}
\usepackage{tikz-cd}
\usetikzlibrary{arrows}

\usepackage{amsthm}

\newtheoremstyle{mystyle}
  {}
  {}
  {\itshape}
  {}
  {\bfseries}
  {.}
  { }
  {}

\theoremstyle{mystyle}

\newtheorem{theorem}{Theorem}[section]

\newtheorem{definition}[theorem]{Definition}
\newtheorem{lemma}[theorem]{Lemma}
\newtheorem{proposition}[theorem]{Proposition}

\newtheorem{example}[theorem]{Example}

\newtheorem{assumption}[theorem]{Assumption}

\usepackage{etoolbox}

\DeclareMathAlphabet{\mathbbold}{U}{bbold}{m}{n}

\title{Topological Linear System Identification via Moderate Deviations Theory
}

\author{Wouter Jongeneel, Tobias Sutter and Daniel Kuhn
\thanks{The authors are with the Risk Analytics and Optimization Chair, École Polytechnique Fédérale de Lausanne, \{wouter.jongeneel, tobias.sutter, daniel.kuhn\}@epfl.ch. This research was supported by the Swiss National Science Foundation under the NCCR Automation, grant agreement~51NF40\_180545.
}%
}
\begin{document}

\maketitle
\thispagestyle{empty}
\pagestyle{empty}

\begin{abstract}
Two dynamical systems are topologically equivalent when their phase-portraits can be morphed into each other by a homeomorphic coordinate transformation on the state space. The induced equivalence classes capture qualitative properties such as stability or the oscillatory nature of the state trajectories, for example. In this paper we develop a method to learn the topological class of an unknown stable system from a single trajectory of finitely many state observations. {Using a moderate deviations principle for the least squares estimator of the unknown system matrix $\theta$,} we prove that the probability of misclassification decays exponentially with the number of observations at a rate that is proportional to the square of the smallest singular value of $\theta$. 
\end{abstract}
\section{Introduction}
\label{sec:intro}
We consider the discrete-time linear time-invariant system
\begin{equation} \label{eq:LTI:system}
x_{t+1} = \theta x_t + w_t, \quad x_0\sim \nu,
\end{equation}
where $x_t\in\mathbb{R}^n$ and $w_t\in\mathbb{R}^n$ denote the state and the exogenous noise at time $t\in\mathbb{N}$, while~$\theta$ represents the system matrix, and $\nu$ stands for the marginal distribution of the initial state~$x_0$. Except for asymptotic stability we assume that nothing is known about $\theta$, and we aim to identify $\theta$ from a single trajectory of states $\{\widehat{x}_t\}_{t=0}^T$ generated by~\eqref{eq:LTI:system}. A simple estimator for~$\theta$ is the least squares estimator
\begin{equation} \label{eq:LS:MDP}
\widehat \theta_T = \left(\textstyle\sum_{t=1}^T \widehat x_t \widehat x_{t-1}^\mathsf{T}\right)\left(\textstyle\sum_{t=1}^T \widehat x_{t-1}\widehat x_{t-1}^\mathsf{T}\right)^{-1},
\end{equation}
which may take any value in $\mathbb{R}^{n\times n}$.
It is therefore possible that $\widehat{\theta}_T$ is unstable even though $\theta$ is stable, in which case the estimator is of limited practical value. Alternative estimators with attractive statistical properties that are {\em guaranteed} to be stable have been proposed in~\cite{ref:Maciejowski:stab,ref:Boots:Nips,ref:Gillis-19,ref:stablesys:2021arXiv}. However, stability is not the only property of~$\theta$ that impacts the qualitative behavior of a linear system; see Figure~\ref{fig:topospringdamper}. As optimal control laws are known to inherit important structural properties from the system matrix~\cite[Theorem~III.1]{ref:Topo2020}, one should aim to find estimators that are structurally equivalent to~$\theta$. 
That is, if the least squares estimator $\widehat{\theta}_T$ is structurally different from $\theta$ itself, in a sense to be made precise later, then implementing optimal linear feedback designed for $\widehat{\theta}_T$  on $\theta$ results in a \textit{closed-loop} system that is structurally different from the predicted closed-loop system, \textit{e.g.}, you predict a damper but get a spring. 
Capturing the correct qualitative behaviour in the scalar case translates to enforcing stability and to the need of estimating the sign of $\theta$ correctly. Generalizing this \textit{qualitative} notion of \textit{topological equivalence} to higher dimensions will be the main subject of Section~\ref{sec:topo:equiv} below.

\textbf{Related work.}
Linear system identification---especially by means of least squares techniques---has a rich history~\cite{ref:OverscheeDeMoor1996,Verhaegen}. In this paper we are, however, not only interested in finding estimators that fall into the vicinity of the unknown true model~$\theta$. In addition, the estimators should display a qualitatively similar behavior as~$\theta$. This requirement relates to some extent to the work on \textit{qualitative identification} pioneered by
\textcite{ref:kuipers1994qualitative}.
More recently, the focus in linear system identification shifted towards ensuring the efficient use of data.
General \textit{informativity} of data is discussed in~\cite{ref:waarde-2019data}, which justifies the identification pipeline for a class of control problems. Moreover, sharp statistical characterizations of the effectiveness of the least squares estimator~\eqref{eq:LS:MDP} are presented in~\cite{ref:Sim_18,ref:sarkar19a}.
These statistical results usually quantify the likelihood that~$\theta$ lies in some \textit{ball} around $\widehat{\theta}_T$. However, the models residing within this ball may be \textit{qualitatively} different.  
Leveraging recent results from the theory of large and moderate deviations~ \cite{ref:Dupuis-97,hollander2008largedeviationis,dembo2009large}, we will be able to characterize the likelihood that the estimated system is qualitatively equivalent to the unknown true system. 

Regarding topological equivalence in the context of linear control systems, \textcite{ref:Willems_Topo} stated in 1980 that ``\textit{{Because of the obvious ... practical importance of these concepts, ... there is no doubt that they will become standard vocabulary among practitioners.}}'' Although \textcite{ref:PoldermanThesis_87} later provided many additional insights, there has been little recent follow-up work on topological properties of control systems. 

\textbf{Contributions.}
A high-level aim of this work is to showcase how topological insights can benefit the control community. More specifically, we establish topological properties of the \textit{reverse $I$-projection} $\mathcal P(\cdot)$ introduced in~\cite{ref:stablesys:2021arXiv}, which projects any matrix in $\mathbb R^{n\times n}$ onto the non-convex set of stable matrices with respect to an information divergence and can be evaluated highly efficiently.
By exploiting tools from moderate deviations theory, we 
characterize here the probability that the reverse $I$-projection $\mathcal{P}(\widehat{\theta}_T)$ of the least-squares estimator~$\widehat{\theta}_T$ is topologically different from~$\theta$. Formally, we show that
\begin{equation*}
    \mathbb{P}(\mathcal{P}(\widehat{\theta}_T) \not\overset{t}{\sim} \theta) \lesssim e^{-\mathcal{O}(\sigma_{\mathrm{min}}(\theta)^2\sqrt{T})},   
\end{equation*}
where `$\overset{t}{\sim}$' denotes topological equivalence. Thus, the probability that $\mathcal{P}(\widehat{\theta}_T)$ misrepresents the topological properties of~$\theta$ decays exponentially with~$T$ at a rate $\propto \sigma_{\mathrm{min}}(\theta)^2$.

 \begin{figure}[t]
     \centering
     \includegraphics[angle=270,scale=0.8]{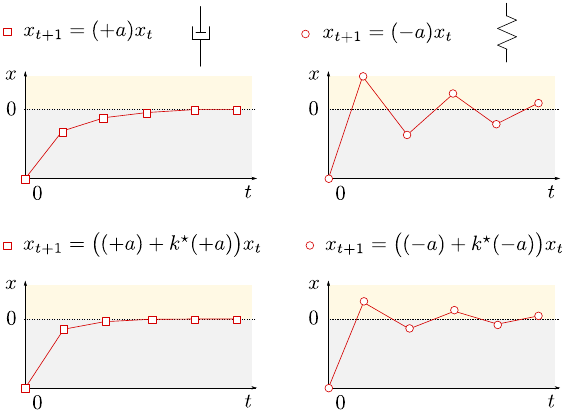}
     \caption{Given any $a\in (0,1)$, the two uncontrolled systems at the top display structurally different behavior, \textit{i.e.}, they respectively represent a damper versus a spring. Applying an optimal LQR controller with feedback gain $k^{\star}$ preserves the structures of the corresponding uncontrolled systems~\cite{ref:Topo2020}.}
     \label{fig:topospringdamper}
 \end{figure}

\textbf{Notation.}
The spectral radius of a matrix $A\in\mathbb{R}^{n\times n}$ is denoted by~$\rho(A)$ and the set of asymptotically stable matrices is denoted by $\Theta=\{\theta\in \mathbb{R}^{n\times n}:\rho(\theta)<1\}$. For a set $\mathcal{D}\subset \mathbb{R}^n$, we use $\mathsf{cl} \, \mathcal{D}$ and $\mathsf{int}\, \mathcal{D}$ to denote the closure and the interior of~$\mathcal{D}$, respectively. For a real sequence $\{a_T\}_{T\in\mathbb{N}}$ the relation $1\ll a_T\ll T$ means that $a_T /T \to 0$ and $a_T\to\infty$ as $T\to\infty$. We denote the real $n$-dimensional general linear group by $\mathsf{GL}(n,\mathbb{R})=\{A\in \mathbb{R}^{n\times n}:\mathrm{det}(A)\neq 0\}$. The sets $\mathsf{GL}^+(n,\mathbb{R})$ and $\mathsf{GL}^-(n,\mathbb{R})$ contain all matrices in $\mathsf{GL}(n,\mathbb{R})$ with a strictly positive or negative determinant.  

\section{Preliminaries}
\label{sec:prelim}
We first discuss the notion of topological equivalence of linear systems in \S~\ref{sec:topo:equiv} and subsequently review the concept of a moderate deviation principle in \S~\ref{sec:moderate:deviations:theory}. 

 \subsection{Topological equivalence of linear dynamical systems}
 \label{sec:topo:equiv}
 This section is mainly inspired by the work of Kuiper and Robbin \cite{ref:Robbin1972,ref:Kuiper_73}. In the following we denote by $f(x)=\theta x$ the state-dependent part of the dynamics~\eqref{eq:LTI:system}, and we refer to the function~$f$ as the \textit{time-one map}. While all time-one maps considered in this paper are linear, our results naturally extend to nonlinear systems. 
 
 \begin{definition}[Topological equivalence]
\label{def:topo_equiv}
Two linear time-one maps $f:\mathcal{X}\to \mathcal{X}$ and $g:\mathcal{Y}\to \mathcal{Y}$ over topological vector spaces $\mathcal{X}$ and $\mathcal{Y}$ are called topologically equivalent (or conjugate), denoted as $f\overset{t}{\sim} g$, if there exists a homeomorphism $\varphi: \mathcal{X}\to \mathcal{Y}$ such that $g\circ \varphi =\varphi \circ f$.
\end{definition}

Recall that a homeomorphism is a continuous bijection with a continjous inverse. We say that two dynamical systems are \textit{topologically equivalent} if their time-one maps are topologically equivalent in the sense of Definition~\ref{def:topo_equiv}. If the two dynamical systems are noise-free ({\em i.e.}, the time-one maps alone determine the dynamics), then they are topologically equivalent if there is a homeomorphism mapping the trajectories of one system onto those of the other~\cite{ref:Robinson}, \cite[Chapter~2]{ref:Kuznetsov}. 
 This means that the state trajectories of topologically equivalent dynamical systems are \textit{qualitatively} identical. To see this, consider two time-one maps $f:\mathcal{X}\to \mathcal{X}$ and $g:\mathcal{Y}\to \mathcal{Y}$ initialized at $x_0$ and $y_0$, respectively, and assume that $f$ and $g$ are topologically equivalent. In this case 
\begin{equation*}
\begin{tikzcd}
x_0 \arrow{r}{f^n} \arrow[swap]{d}{\varphi} & f^n(x_0) \arrow{d}{\varphi} \\
y_0 \arrow{r}{g^n} & g^n(y_0)
\end{tikzcd}
\end{equation*}
commutes for any $n\in \mathbb{N}$ because $f=\varphi\circ g \circ \varphi^{-1}$~\cite{ref:Kuznetsov}. Note that this argument holds indeed for any pair $(x_0,y_0)$ of initial states, that is, $\varphi$ essentially constitutes a homeomorphic coordinate change. A change  of a dynamical system that destroys topological equivalence is called a \textit{bifurcation}.

 
If two linear time-one maps of the form $f(x)=Fx$ and $g(y)=Gy$ parametrized by the matrices $F$ and $G$ are topologically equivalent, then, by slight abuse of notation, we write $F\overset{t}{\sim} G$. We emphasize that topological equivalence generalizes the more common \emph{linear equivalence}, where $\varphi(x)=Zx$ is a linear isomorphism parametrized by a matrix $Z$. In this case the time-one maps $f(x)=Fx$ and $g(y)=Gy$ with $G=ZAZ^{-1}$ are linearly equivalent. 
To gain some intuition for Definition~\ref{def:topo_equiv}, consider a scalar system of the form $f(x)=a x$ for some~$a\in\mathbb R$. Figure~\ref{fig:scalar:interval} shows that this system admits seven equivalence classes with respect to the topological equivalence relation at hand, which correspond to seven intervals in~$\mathbb{R}$. This example also shows that one should {\em not} expect both $\varphi$ and $\varphi^{-1}$ to be differentiable. Indeed, assuming $\varphi$ to be a diffeomorphism implies that $\varphi$ must be linear~\cite[Proposition~6.1]{ref:Robinson}. Note also that any homeomorphism on $\mathbb{R}$ is necessarily monotone, which confirms our earlier insight that the damper can never be mapped to the spring in Figure~\ref{fig:topospringdamper}. 
\begin{figure}
    \centering
    \includegraphics[angle=270,scale=0.75]{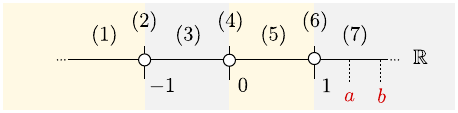}
    \caption{Visualization of the seven topological equivalence classes of a scalar linear system ~\cite[Proposition~1.5]{ref:Kuiper_73}. Note that if the time-one maps $f(x)=ax$ and $g(y)=by$ belong to the same non-degenerate class ({\em i.e.}, $a$ and $b$ both belong to interval $(1)$, $(3)$, $(5)$ or $(7)$), then $g=\varphi\circ f \circ \varphi^{-1}$ for $\varphi(x) = x|x|^{c-1}$ with $c=\log(|b|)/\log(|a|)$. For example, $f(x)=2x$ and $g(y)=8y$ belong to class $(7)$ and are related by the smooth map $\varphi(x)=x^3$ with inverse $\varphi^{-1}=x^{1/3}$ \cite{ref:Kuiper_73}.}
    \label{fig:scalar:interval}
\end{figure}

 Asymptotic stability or topological invariants such as orientation provide conditions for the topological equivalence of two time-one maps~\cite{ref:Kuiper_73}.

\begin{definition}[Orientation]
We call a linear invertible time-one map $f(x)=Fx$ orientation-preserving if $\mathrm{det}(F)>0$, and we denote by $\mathrm{or}(f)$ the sign of $\mathrm{det}(F)$. 
\end{definition}

Note that $\mathsf{GL}^+(n,\mathbb{R})$ and $\mathsf{GL}^-(n,\mathbb{R})$ represent the components of $\mathsf{GL}(n,\mathbb{R})$ that contain all invertible matrices~$F$ whose associated time-one maps $f(x)=Fx$ satisfy $\mathrm{or}(f)=1$ and $\mathrm{or}(f)=-1$, respectively. Note also that the sign of the (signed) volume of the unit hypercube is invariant under an orientation-preserving map. For additional information on the concept of orientation see for example~\cite[Chapter~6]{MarsdenTensors}.
The orientation of a map is also known to be a \emph{topological invariant}, that is, if two invertible maps~$f$ and~$g$ are topologically equivalent, then $\mathrm{or}(f)=\mathrm{or}(g)$ \cite[Chapters~6 \& 10]{Lee1}. For brevity, we do not generalize the orientation operator to linear maps that fail to be invertible. The following lemma provides a convenient tool for checking topological equivalence.

\begin{lemma}[Topological equivalence of asymptotically stable systems~{\cite[Theorem~9.2]{ref:Robinson}} ]
\label{lem:main:topo:equiv}
Assume that $f(x)=Fx$ and $g(y)=Gy$ are asymptotically stable linear isomorphisms on $\mathbb{R}^n$. Then, we have~$f\overset{t}{\sim} g$ if and only if there exists a continuous path $X(t)$ in $\mathsf{GL}(n,\mathbb{R})$ parametrized in~$t\in [0,1]$ with $X(0)=F$ and $X(1)=G$.
\end{lemma}
As $\mathsf{GL}(n,\mathbb{R})$ can be decomposed into $\mathsf{GL}^+(n,\mathbb{R})$ and $\mathsf{GL}^-(n,\mathbb{R})$, both of which are connected, Lemma~\ref{lem:main:topo:equiv} implies that these $f$ and $g$ are topologically equivalent if and only if the signs of the determinants of $F$ and $G$ match. By slight abuse of notation, we henceforth define the orientation $\mathrm{or}(F)$ of an invertible matrix~$F$ as the sign of $\det(F)$.


\subsection{Moderate deviations theory}
\label{sec:moderate:deviations:theory}
Throughout the paper we assume that all random objects are defined on a measurable space~$(\Omega, \mathcal{F})$ equipped with a probability measure~$\mathbb{P}_\theta$ parametrized by the (asymptotically stable) system matrix~$\theta\in\Theta$ from~\eqref{eq:LTI:system}. We denote the expectation operator with respect to $\mathbb{P}_\theta$ by $\mathbb{E}_{\theta}[\cdot]$. From now on we impose the following assumption borrowed from~\cite{ref:stablesys:2021arXiv}.


\begin{assumption}[Linear system]\label{ass:linear:sys} The following hold.
\begin{enumerate}[(i)]
    \itemsep0em 
	\item \label{ass:lin:sys:stability} The system~\eqref{eq:LTI:system} is asymptotically stable, \textit{i.e.,} $\theta\in\Theta$.
	\item \label{ass:lin:sys:exogeneous:noise} For each $\theta\in\Theta$ the disturbances $\{w_t\}_{t\in\mathbb{N}}$ are independent and identically distributed (i.i.d.) and independent of $x_0$ under $\mathbb{P}_\theta$. The disturbances are unbiased ($\mathbb{E}_{\theta}[w_t]=0$) and non-degenerate ($S_w=\mathbb{E}_{\theta}[w_t w_t^\mathsf{T}]\succ 0$), and their probability density is everywhere positive.
\end{enumerate}
\end{assumption}

Assumption~\ref{ass:linear:sys} implies that the linear system~\eqref{eq:LTI:system} admits an invariant distribution~$\nu_\theta$~\autocite[\S~10.5.4]{ref:meyn-09} with zero mean and covariance matrix~$S_\theta$, which is given by the unique positive definite solution of the discrete Lyapunov equation 
\begin{equation}\label{eq:Lyapunov}
S_\theta = \theta S_\theta \theta^\mathsf{T} + S_w,
\end{equation}
see, {\em e.g.}, \autocite[\S~6.10\,E]{ref:Antsaklis-06}. Next, we describe a method to characterize the probability of the least squares estimator~$\widehat{\theta}_T$ deviating from $\theta$ by a prescribed threshold. To this end, we denote by~$\Theta'=\mathbb R^{n\times n}$ the space of all estimator realizations that are possible in view of Assumption~\ref{ass:linear:sys}, and we use the discrepancy function $I:\Theta'\times \Theta\rightarrow [0,\infty]$ with
\begin{equation} \label{eq:rate:function:AR:MDP}
	I(\theta^\prime,  \theta) = \tfrac{1}{2} \mathrm{tr}\left(S_w^{-1}(\theta^\prime - \theta) S_\theta (\theta^\prime - \theta)^\mathsf{T}\right)
\end{equation}
to quantify the difference between an estimator realization~$\theta'\in\Theta'$ and the system matrix~$\theta\in\Theta$; see~\cite{ref:stablesys:2021arXiv}.
Here, the invariant state covariance matrix~$S_\theta
$ is defined as in~\eqref{eq:Lyapunov}. Note that~$S_\theta
$ and thus also~$I(\theta^\prime,  \theta)$ diverge as~$\theta$ approaches the boundary of~$\Theta$ and becomes unstable. Note also that since $S_w\succ 0$ and hence $S_{\theta}\succ 0$, $I(\theta^\prime,  \theta)$ vanishes if and only if $\theta'=\theta$. In this sense $I$ behaves like a distance. 
Note, however, that $I(\theta^\prime,  \theta)$ is not symmetric in~$\theta$ and~$\theta'$.

Next we recall the notions of a {\em rate function} and a {\em moderate deviation principle} (MDP) such that we can review the key results of~\cite{ref:stablesys:2021arXiv}. For a comprehensive introduction to moderate deviations theory we refer to~\autocite{hollander2008largedeviationis, dembo2009large}. 
\begin{definition}[Rate function]
	\label{def:rate_function:original}
	We call $I:\Theta'\times \Theta\rightarrow [0,\infty]$ a rate function if it is lower semi-continuous in~$\theta'$.
\end{definition}

\begin{definition}[Moderate deviation principle]
	\label{def:MDP:general}
	A sequence of estimators $\{\widehat\vartheta_T\}_{T\in\mathbb{N}}$ is said to satisfy a moderate deviation principle with rate function $I$ if for every sequence $\{a_T\}_{T\in\mathbb{N}}$ of real numbers with $1\ll a_T\ll T$, for every Borel set $\mathcal{D}\subset \Theta'$ and for every $\theta\in\Theta$ we have
	\begin{subequations}
		\label{eq:ldp_exponential_rates}
		\begin{align}
		\label{eq:ldp_exponential_rates_lb}
		-\inf_{\theta' \in \mathsf{int}\,{\mathcal{D}}} \, I(\theta', \theta) \leq& \liminf_{T\to \infty}~\frac{1}{a_T} \log \mathbb{P}_{\theta}\left( \widehat \vartheta_T \in \mathcal{D} \right) \\
		\leq& \limsup_{T\to \infty}~\frac{1}{a_T} \log \mathbb{P}_\theta\left( \widehat \vartheta_T \in \mathcal{D} \right)\\ 
		 \leq &-\inf_{\theta' \in \mathsf{cl}\, {\mathcal{D}}} \,  I(\theta', \theta).
		\label{eq:ldp_exponential_rates_ub}
		\end{align}
	\end{subequations}
\end{definition}
If $I(\theta',\theta)$ is continuous (instead of merely lower semi-continuous) in $\theta'$ and if~$\mathsf{int}\,\mathcal{D}$ is dense in $\mathcal{D}$, then the infima in~\eqref{eq:ldp_exponential_rates_lb} and~\eqref{eq:ldp_exponential_rates_ub} match, and all inequalities in~\eqref{eq:ldp_exponential_rates} reduce to equalities. Then, \eqref{eq:ldp_exponential_rates} implies that~$\mathbb{P}_\theta(\widehat \vartheta_T\in\mathcal{D})= e^{-r a_T+o(a_T)}$, where~$r=\inf_{\theta' \in \mathcal{D}} I(\theta', \theta)$ quantifies the $I$-distance of the system matrix~$\theta$ from the set~$\mathcal{D}$, see Figure~\ref{fig:Idist}. 
 Thus, the distance $r$ coincides with the decay rate of the probability that~$\widehat \vartheta_T$ materializes in~$\mathcal{D}$, while~$\{a_T\}_{T\in\mathbb{N}}$ can be viewed as the speed of convergence. The condition~$1\ll a_T\ll T$ is satisfied, for example, if~$a_T=\sqrt{T}$, $T\in\mathbb{N}$.

\begin{figure}
    \centering
    \includegraphics[scale=0.75,angle=270]{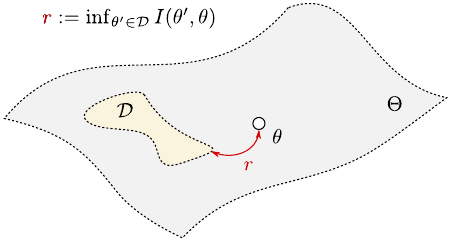}
    \caption{The $I$-distance between $\theta$ and some set $\mathcal{D}\not\ni \theta$ quantifies the rate at which the likelihood of an estimator $\widehat{\theta}_T$ falling into $\mathcal{D}$ decays, \textit{i.e.,} $\mathbb{P}_\theta(\widehat \theta_T\in\mathcal{D})\approx e^{-r a_T}$. }
    \label{fig:Idist}
\end{figure}

It is unknown whether the least squares estimators $\widehat \theta_T$ defined in~\eqref{eq:LS:MDP} satisfy an MDP for $n>1$. However, the \textit{transformed} least squares estimators defined~as
\begin{equation} \label{eq:transformed:LSE}
    \widehat \vartheta_T(\widehat \theta_T,\theta)
    =\sqrt{T/a_T}(\widehat{\theta}_T-\theta)+\theta
\end{equation}
are known to satisfy an MDP for any~$n\in\mathbb N$, where the underlying rate function given by the discrepancy function~\eqref{eq:rate:function:AR:MDP}. Whenever there is no risk of confusion, we drop the explicit dependence of~$\widehat \vartheta_T$ on $\widehat{\theta}_T$ and $\theta$. In order to formally introduce the advertised MDP, we impose another standard regularity condition borrowed from~\cite{ref:stablesys:2021arXiv}, which is again assumed to hold throughout the remainder of the paper. 

\begin{assumption}[Light-tailed noise and stationarity]\label{ass:model:MDP:regularity} The following hold for every $\theta\in\Theta$.
\begin{enumerate}[(i)]
	\item \label{ass:control:regularity:i} The disturbances $\{w_t\}_{t\in\mathbb{N}}$ are
	light-tailed, \textit{i.e.}, there exists $\alpha >0$ with $\mathbb{E}_{\theta}[e^{\alpha \|w_t\|^2}]< \infty$ for all $t\in\mathbb{N}$.
	\item \label{ass:control:regularity:ii} The initial distribution $\nu$ coincides with the invariant distribution $\nu_\theta$ of the linear system \eqref{eq:LTI:system}.
\end{enumerate}
\end{assumption}

We can now formally state the MDP for the estimators~\eqref{eq:transformed:LSE}.

\begin{proposition}[Moderate deviation principle~{\cite[Proposition 3.4]{ref:stablesys:2021arXiv}}]
\label{prop:MDP:closed:loop}
	If $\{a_T\}_{T\in\mathbb{N}}$ is a real sequence with $1\ll a_T\ll T$, then the transformed least squares estimators
	$\{\widehat \vartheta_T\}_{T\in\mathbb{N}}$ defined in~\eqref{eq:transformed:LSE} satisfy an MDP with rate function~\eqref{eq:rate:function:AR:MDP}.
\end{proposition}


Proposition~\ref{prop:MDP:closed:loop} is attractive due to the generality of MDPs, which provide tight bounds on the (asymptotic) probability of \textit{any} Borel set of estimator realizations. A simple direct application of Proposition~\ref{prop:MDP:closed:loop} is described below.

\begin{example}[Scalar system identification and noise invariance]
\label{ex:sysid:noise:mag}
\upshape{
Consider a scalar system with $S_w=\sigma_w^2>0$. As shown in~\cite[Example 3.5]{ref:stablesys:2021arXiv}, Proposition~\ref{prop:MDP:closed:loop} implies that
\begin{align*}
    \mathbb{P}_{\theta}(|\widehat{\theta}_T-\theta|&\!> \varepsilon \sqrt{a_T/T})\\
    &= \textstyle \exp\left( -\frac{1}{2}\varepsilon^2\,a_T/(1-\theta^2) +o(a_T)\right)
\end{align*}
for any $\varepsilon>0$ and $T\in\mathbb{N}$. 
Thus, the decay rate on right hand side of the above expression is independent of the noise intensity $\sigma_w^2$. If $n>1$, define $Z(\theta',\theta)=(\theta'-\theta) \otimes ( \theta'-\theta)(I_{n^2}-\theta \otimes \theta)^{-1}$, where $\otimes$ denotes the Kronecker product. 
By~\cite[p.~265]{ref:Hamilton-94}, the rate function~\eqref{eq:rate:function:AR:MDP} can then be recast as
\begin{align*}
	I(\theta',\theta)&=\tfrac{1}{2}\mathrm{vec}(S_w^{-1})^{\mathsf{T}}Z(\theta',\theta)\mathrm{vec}(S_w),
	\end{align*}
where $\mathrm{vec}(S_w)\in \mathbb{R}^{n^2}$ represents the vector obtained by stacking the columns of $S_w$ on top of each other. This representation reveals that the rate function is indeed invariant under scaling of the noise covariance matrix, that is, it is independent of the overall noise level for all $n\in\mathbb N$.
} \hfill\qed
\end{example}

We finally highlight that the rate function~\eqref{eq:rate:function:AR:MDP} can be used to construct a reverse $I$-projection defined through
\begin{equation}
\label{equ:rev:I:proj}
    \mathcal{P}(\theta') \in \arg\min_{\theta\in\Theta} I(\theta',{\theta}),
\end{equation}
which maps any point~$\theta'\in\Theta'$ to the nearest point in the non-convex set~$\Theta$ of asymptotically stable matrices with respect to the $I$-distance; see~\cite{ref:stablesys:2021arXiv}. The reverse $I$-projection $\mathcal{P}(\widehat{\theta}_T)$ of the least squares estimator $\widehat{\theta}_T$ provides an estimator for~$\theta$ that is stable by construction, has desirable statistical properties and can be efficiently computed. Indeed, one can show that
\begin{equation}
\label{equ:rev;I:proj:comp}
    \mathcal{P}(\theta') = \theta' + \mathsf{dlqr}(\theta',I_n,I_n,(2\delta S_w)^{-1}) + \mathcal{O}(\delta^p)
\end{equation}
for some $p\geq 1$,
where $\mathsf{dlqr}(\cdot)$ denotes the standard discrete-time LQR routine,\footnote{\url{https://juliacontrol.github.io/ControlSystems.jl/latest/examples/example/\#LQR-design}}
see~\cite[\S 3.3]{ref:stablesys:2021arXiv}. In addition, the reverse $I$-projection preserves orientation, \textit{i.e.}, $\mathrm{or}(\mathcal{P}(\theta')) = \mathrm{or}(\theta')$ for any $\theta'\in \mathsf{GL}(n,\mathbb{R})$, see, \textit{e.g.}, \cite[Corollary 3.12]{ref:stablesys:2021arXiv} and~\cite{ref:Topo2020}. In fact, the numerical approximation of $\mathcal{P}(\theta')$ on the right hand side of~\eqref{equ:rev;I:proj:comp} without the error term~$\mathcal O(\delta^p)$ is also asymptotically stable and preserves orientation for any~$\delta>0$. Due to its desirable statistical and computational properties, our proposed approach to estimate the topological class of~$\theta\in\Theta$ will critically rely on the reverse $I$-projection~$\mathcal{P}(\widehat{\theta}_T)$. 


\section{Main results}
\label{sec:main}
To aid the presentation we assume from now on without much loss of generality that $\theta$ is invertible. We are now ready to demonstrate that the MDP of  Proposition~\ref{prop:MDP:closed:loop} allows us via the reverse $I$-projection~$\mathcal{P}(\widehat{\theta}_T)$ to derive sharp bounds on the decay rate of the probability of the event $\mathcal{P}(\widehat{\theta}_T)\not\overset{t}{\sim}\theta$.  
Recall from Lemma~\ref{lem:main:topo:equiv} that the two stable and ($\mathbb P_\theta$-almost surely) invertible matrices $\mathcal{P}(\widehat{\theta}_T)$ and $\theta$ are topologically equivalent if and only if they have the same orientation. Recall also that $\mathrm{or}(\mathcal{P}(\widehat{\theta}_T)) = \mathrm{or}(\widehat{\theta}_T)$ because the reverse $I$-projection preserves orientation. Checking whether $\mathcal{P}(\widehat{\theta}_T)$ is topologically equivalent to $\theta$ is thus tantamount to checking whether the determinants of $\widehat{\theta}_T$ and $\theta$ have the same signs. 

\begin{theorem}[Probability of misclassification]
\label{thm:topo:equiv:rate}
Assume that $\theta \in \Theta\cap \mathsf{GL}(n,\mathbb{R})$, $\{\widehat{\theta}_T\}_{T\in \mathbb{N}}$ are the least squares estimators~\eqref{eq:LS:MDP} and $\{a_T\}_{T\in \mathbb{N}}$ is a sequence with $1\ll a_T \ll T$. If $S_{\theta^{\circ}}=S_w^{-1/2}S_\theta S_w^{-1/2}$ and $r=\frac{1}{2}\lambda_{\mathrm{min}}(S_{\theta^{\circ}}-I_n)$, then
\begin{subequations}
\label{eq:main-results}
 \begin{equation}
\label{equ:or:prob}
    \limsup_{T\to \infty}\frac{1}{a_T}\log \mathbb{P}_{\theta}\big( \mathrm{or}(\widehat{\theta}_T) \not = \mathrm{or}(\theta) \big) \leq - r, 
\end{equation}   
\begin{equation}
\label{equ:thm:prob}
    \limsup_{T\to \infty}\frac{1}{a_T}\log \mathbb{P}_{\theta}\big( \mathcal{P}(\widehat{\theta}_T)\not\overset{t}{\sim} \theta    \big) \leq - r. 
\end{equation}
\end{subequations}
\end{theorem}
\begin{proof}
As for~\eqref{equ:or:prob}, note that $\widehat{\theta}_T\in \mathsf{GL}(n,\mathbb{R})$ $\mathbb{P}_{\theta}$-almost surely for all sufficiently large $T$, and therefore we have
\begin{align*}
    \mathbb{P}_{\theta}\big( \mathrm{or}(\widehat{\theta}_T) \not = \mathrm{or}(\theta) \big) & =\mathbb{P}_{\theta}\big( \mathrm{or}(\widehat{\theta}_T) = -\mathrm{or}(\theta) \big)\\
    & =\mathbb{P}_{\theta}\big(\exists G\in \mathsf{GL}^- (n,\mathbb{R}): \widehat{\theta}_T=G\theta\big)\\
    & =\mathbb{P}_{\theta}\big(\widehat{\vartheta}_T\in\mathcal D_T(\theta)\big)
\end{align*}
where the second equality holds because any invertible matrices~$\widehat{\theta}_T$ and~$\theta$ whose determinants have opposite signs satisfy~$\widehat{\theta}_T=G\theta$ for some $G\in \mathsf{GL}^-(n,\mathbb{R})$. The third equality follows from the definition of the transformed least squares estimators $\{\widehat{\vartheta}_T\}_{T\in \mathbb{N}}$ in~\eqref{eq:transformed:LSE} and the construction of the set
\begin{equation*}
    \begin{aligned}
    \mathcal{D}_T(\theta) = \{& \sqrt{(T/a_T)}(G - I_n)\theta+\theta: G\in \mathsf{GL}^-(n,\mathbb{R})\}. 
    \end{aligned}
\end{equation*}
As this set is time-dependent, we cannot directly use it. That is, it is not admissible in the sense of Definition~\ref{def:MDP:general}. To sidestep this complication, we consider instead the larger set
\begin{equation*}
    \mathcal{D}(\theta) = \bigcup_{T\in \mathbb{N}}\mathcal{D}_T(\theta).
\end{equation*}
The above reasoning then implies that
\begin{align*}
    &\limsup_{T\to \infty}\frac{1}{a_T}\log \mathbb{P}_{\theta}\big(\mathrm{or}(\widehat{\theta}_T) \not = \mathrm{or}(\theta) \big)\\
    \leq& \limsup_{T\to \infty}\frac{1}{a_T}\log \mathbb{P}_{\theta}\big(\widehat{\vartheta}_T \in\mathcal D(\theta) \big)\leq - \inf_{\theta'\in \mathsf{cl}\,\mathcal{D}(\theta)}I(\theta',\theta),
\end{align*}
where the first inquality holds because $\mathcal{D}_T(\theta)\subset \mathcal{D}(\theta)$ for all~$T\in\mathbb N$, while the second inequality follows from the MDP established in Proposition~\ref{prop:MDP:closed:loop}. In the remainder of the proof we derive an analytical lower bound on the minimization problem on the right hand side of the above expression. To this end, assume first that $S_w=I_n$. Evaluating the rate function~\eqref{eq:rate:function:AR:MDP} at an arbitrary $\theta'\in \mathcal{D}(\theta)$ then yields \begin{align}
\label{equ:I:eval}
    I(\theta',\theta) = \frac{T}{2 a_T}\mathsf{tr}\left((G-I_n)\theta S_{\theta} \theta^{\mathsf{T}}(G-I_n)^{\mathsf{T}} \right)
\end{align}
for some $G\in \mathsf{GL}^-(n,\mathbb{R})$ and some $T\in \mathbb{N}$, and
the Lyapunov equation~\eqref{eq:Lyapunov} implies that $\theta S_{\theta} \theta^{\mathsf{T}}=S_{\theta}-I_n$. In addition, our assumptions about the sequence $\{a_T\}_{T\in \mathbb{N}}$ imply that~$T/a_T\ge 1$. 
We may thus conclude that
\begin{align*}
    \min_{\theta'\in\mathsf{cl}\,\mathcal{D}(\theta)} & I(\theta',\theta) \\ &\geq
    \inf_{\det(G)\le 0} \tfrac{1}{2} \mathsf{tr}\left((G-I_n) (S_\theta-I_n) 
    (G-I_n)^{\mathsf{T}} \right) \\
    &\geq \tfrac{1}{2}(\lambda_{\mathrm{min}}(S_{\theta})-1) \inf_{\det(G)\le 0}\| G-I_n\|_\mathsf{F}^2 \\
    &=\tfrac{1}{2}(\lambda_{\mathrm{min}}(S_{\theta})-1),
\end{align*}
where the first inequality holds because $\det(G)\le 0$ for every $G\in \mathsf{cl}\,\mathsf{GL}^-(n,\mathbb{R})$, the second inequality uses the bound $\mathsf{tr}(AB) \geq \sigma_\mathrm{min}(A) \mathsf{tr}(B)$ for any $A,B\succeq 0$, and the third inequality follows from the Eckart-Young Theorem~\cite[Theorem~2.4.8]{ref:Gvl}. One can actually show that the second inequality is tight, but this is not needed to prove the theorem. 
This establishes~\eqref{equ:or:prob} for~$S_w=I_n$. If $S_w\succ 0$ is arbitrary, one may first apply a change of coordinates $x^{\circ}=S_w^{-1/2}x$, under which the noise covariance matrix simplifies to~$I_n$, while the system matrix and the invariant state covariance matrix become $\theta^{\circ}=S_w^{-1/2}\theta S_w^{1/2}$ and $S_{\theta^{\circ}}=S_w^{-1/2}S_{\theta}S_w^{-1/2}$, respectively. As $S_w\succ 0$, we further have $\theta^{\circ}\overset{t}{\sim}\theta$. Applying the results of the first part of the proof to the transformed system finally yields~\eqref{equ:or:prob}. As for~\eqref{equ:thm:prob}, recall from Sections~\ref{sec:topo:equiv} and~\ref{sec:moderate:deviations:theory} that $\mathcal{P}(\theta')$ is asymptotically stable and that $\mathrm{or}(\theta')=\mathrm{or}(\mathcal{P}(\theta'))$ for any $\theta'\in \mathsf{GL}(n,\mathbb{R})$. Therefore, we can focus on orientation solely, \textit{i.e.,} if $\mathrm{or}(\theta')=\mathrm{or}(\theta)$ then $\mathcal{P}(\theta')\overset{t}{\sim}\theta$. Then again, as $\mathbb{P}_{\theta}\big(\widehat{\theta}_T\not\in \mathsf{GL}(n,\mathbb{R})\big)=0$ for sufficiently large $T$, the claim follows from~\eqref{equ:or:prob}.  
\end{proof}
The rate $r=\frac{1}{2}\lambda_{\mathrm{min}}(S_{\theta^{\circ}}-I_n)$ established in Theorem~\ref{thm:topo:equiv:rate} is non-trivial (strictly positive) for any $\theta\in \Theta\cap \mathsf{GL}(n,\mathbb{R})$ as
\begin{equation*}
 S_{\theta^{\circ}} = S_w^{-1/2}S_{\theta}S_w^{-1/2} \overset{\eqref{eq:Lyapunov}}{=} S_w^{-1/2}\sum^{\infty}_{k=1}\theta^k S_w (\theta^k)^{\mathsf{T}}S_w^{-1/2} + I_n.
\end{equation*}

\subsection{Implications of Theorem~\ref{thm:topo:equiv:rate}}
{
As the rate~$r$ derived in Theorem~\ref{thm:topo:equiv:rate} is a function of~$\theta$, the properties of the unknown system matrix~$\theta$ determine the likelihood of topological misclassification.
In high-performance applications where some eigenvalues of $\theta$ are close to~$0$, for example, the sign of the determinant of~$\theta$ and therefore the topological class of the underlying system are difficult to estimate.
As such applications are usually safety-critical, however, inferring the correct topological class is of utmost importance. In the context of Figure~\ref{fig:topospringdamper}, designing a controller tailored to a damper might be detrimental if the true system is a spring. Theorem~\ref{thm:topo:equiv:rate} indicates that designing a system for high performance in the sense of constructing some eigenvalues of~$\theta$ to be close to~$0$  
is in conflict with the reliable and fast identification of the system's topological class. 
}

 Quantitative notions of stability and controllability, which strengthen the standard qualitative notions of stability and controllability, respectively, offer further insights into~\eqref{equ:thm:prob}.
 
\begin{definition}[Strong controllability~\cite{ref:Cohen18b}]
\label{def:strong:contr}
A pair $(A,B)$ is called $(\ell,\nu)$-strongly controllable for $\ell\in\mathbb N$ and $\nu>0$ if the matrix $C_{\ell}=\begin{pmatrix}B\,AB\,\cdots\,A^{\ell-1}B\end{pmatrix}$ satisfies $\sigma_{\mathrm{min}}(C_{\ell})\geq \nu$.
\end{definition}

If $A$ and $B$ parametrize the deterministic system $x_{t+1}=Ax_t+Bu_t$ and $(A,B)$ is $(\ell,\nu)$-strongly controllable, 
then, a small $\nu$ indicates that for two points $x_a,x_b$ being close, the inputs to reach them from $x_0$ might be wildly different.
\begin{definition}[Strong stability~\cite{ref:Cohen18b}]
\label{def:strong:stab}
A square matrix $A$ is $(\kappa,\gamma)$-strongly stable for $\kappa>0$ and $\gamma\in (0,1]$ if $A=HLH^{-1}$ for $L$ and $H$ such that $\|L\|_2\leq 1-\gamma$ and $\|H\|_2\|H^{-1}\|_2\leq \kappa$. 
\end{definition}
One readily verifies that any $(\kappa,\gamma)$-strongly stable matrix is asymptotically stable. Conversely, by Lyapunov stability theory, for any asymptotically stable matrix $A$ there exists $P\succ 0$ such that $P-A^{\mathsf{T}}P A \succ 0$. Thus, the transformed matrix $A'=P^{1/2}A P^{-1/2}$ satisfies $I_n-(A')^{\mathsf{T}}A'\succ 0$, which in turn implies that $\|A'\|_2< 1$. Hence, $A=HLH^{-1}$ with $H=P^{-1/2}$ and $L=A'$, and $A$ is $(\kappa,\gamma)$-strongly stable with $\gamma=1-\|A'\|_2$ and $\kappa$ being proportional to the condition number of $P^{1/2}$. The next lemma relates the rate~$r$ of Theorem~\ref{thm:topo:equiv:rate} to strong controllability and stability properties of the pair $(\theta^{\circ},I_n)$.


\begin{lemma}[Bounding the rate of Theorem~\ref{thm:topo:equiv:rate}]
\label{lem:Stheta:bound}
If the pair $(\theta^{\circ},I_n)$ is $(\ell,\nu)$-strongly controllable and $\theta^{\circ}$ is $(\kappa,\gamma)$-strongly stable, then
\begin{equation}
    \label{equ:Stheta:bound}
    \nu^2 \leq \lambda_{\mathrm{min}}(S_{\theta^{\circ}})\leq (\kappa^2)/(2\gamma -\gamma^2). 
\end{equation}
\end{lemma} 
\begin{proof}
The matrix $C_{\ell}$ introduced in Definition~\ref{def:strong:contr} with $A=\theta^\circ$ and $B=I_n$ satisfies $C_{\ell}C_{\ell}^{\mathsf{T}} = \sum^{\ell-1}_{i=0}(\theta^\circ)^i ((\theta^\circ)^i)^{\mathsf{T}}$. Thus, the controllability Gramian $W=\lim_{\ell\to \infty}C_{\ell}C_{\ell}^{\mathsf{T}}$ coincides with~$S_{\theta^\circ}$ as defined in Theorem~\ref{thm:topo:equiv:rate}. As $(\theta^{\circ},I_n)$ is $(\ell,\nu)$-strongly controllable, we may thus conclude that $\lambda_{\mathrm{min}}(S_{\theta^{\circ}})\geq \lambda_{\mathrm{min}}(C_\ell C_\ell^{\mathsf{T}} )\geq \nu^2$. An upper bound on $\lambda_{\mathrm{min}}(S_{\theta^{\circ}})$ can be 
obtained from the $(\kappa,\gamma)$-strong stability of~$\theta^\circ$, which implies via~\cite[Lemma 3.3]{ref:Cohen18b} that $\mathsf{tr}(S_{\theta^{\circ}})\leq (\kappa^2/\gamma)\mathsf{tr}(I_n)$ and hence $\lambda_{\mathrm{min}}(S_{\theta^{\circ}})\leq (\kappa^2/\gamma)$. Looking at the proof of~\cite[Lemma 3.3]{ref:Cohen18b}, this can be sharpened to~\eqref{equ:Stheta:bound}.
\end{proof}

The bound~\eqref{equ:Stheta:bound} {is in line with the folklore wisdom} in system identification that a slow ($\gamma$ small), yet well-excited ($\nu$ large) system is desirable. Next, we highlight a few other properties of the problem parameters $\theta$ and $S_w$ one might be able to manipulate in order to increase the rate in Theorem~\ref{thm:topo:equiv:rate} {(even though this rate depends on $\theta$ and is thus unobservable)}.

\subsubsection{Tuning $\sigma_{\mathrm{min}}(\theta^{\circ})$ and the noise covariance matrix} 
    Using 
    the series representation of $S_{\theta^{\circ}}$ one can 
    show that
    \begin{equation*}
        S_{\theta^{\circ}}\succeq I_n + \frac{1}{1-\sigma_{\mathrm{min}}(\theta^{\circ})^2}\theta^{\circ} (\theta^{\circ})^{\mathsf{T}},
    \end{equation*}
    see~\cite{ref:CHLeebounds}, where $\sigma_{\mathrm{min}}(\theta^{\circ})\leq \rho(\theta^{\circ})<1$. Since also $S_{\theta^\circ}-I_n\succeq\theta^\circ (\theta^\circ)^\top$ it follows that $        \tfrac{1}{2} \lambda_\mathrm{min}(S_{\theta^\circ}-I_n)
         \geq \tfrac{1}{2}  \sigma_\mathrm{min}(\theta^\circ)^2.$
     Hence, an increase in $\sigma_{\mathrm{min}}(\theta^{\circ})$ improves the rate $r$ from Theorem~\ref{thm:topo:equiv:rate}.
    Next, assume that $\theta$ is diagonalizable, \textit{i.e.}, $\theta=V\Lambda V^{-1}$ for some diagonal matrix~$\Lambda$ and invertible matrix~$V$. As $\sigma_{\mathrm{min}}(\theta^{\circ})\leq \lambda_{\mathrm{min}}(\theta^{\circ})=\lambda_{\mathrm{min}}(\theta)$, the preferred noise covariance matrix for which $\sigma_{\mathrm{min}}(\theta^{\circ})$ matches the bound $\lambda_{\mathrm{min}}(\theta)$ (which is independent of~$S_w$) 
    is given by $S_w=\alpha VV^{\mathsf{T}}$ for any $\alpha>0$. 
    As already pointed out in Example~\ref{ex:sysid:noise:mag} the magnitude of $S_w$ is not important, its principal axes are.
    
    \subsubsection{Tuning $\sigma_{\mathrm{min}}(\theta)$} The rate $\mathcal{O}(-\sigma_{\mathrm{min}}(\theta)^2 \sqrt{T})$ from the introduction follows as $\sigma_{\mathrm{min}}(\theta)\leq \sigma_{\mathrm{min}}(\theta^{\circ}) \sigma_{\mathrm{max}}(S_w^{1/2})/\sigma_{\mathrm{min}}(S_w^{1/2})$. In a system-theoretic context, as $\sigma_{\mathrm{min}}(\theta)=\inf_{\|x\|_2=1}\|\theta x\|_2$, an increase in $\sigma_{\mathrm{min}}(\theta)$ reduces the contraction rate of~\eqref{eq:LTI:system}. 
    
    \subsubsection{Minimizing interconnections}\label{interconnections} Consider a separable and an interconnected system with
    \begin{equation*}
        \theta_1 = \begin{pmatrix} Y & 0 \\ 0 & Y \end{pmatrix}\quad \text{and} \quad \theta_2 = \begin{pmatrix} Y & I_n \\ 0 & Y \end{pmatrix}
    \end{equation*}
    for some $Y\in \Theta$, respectively. 
    One can show that $\sigma_{\mathrm{min}}(\theta_1)\geq \sigma_{\mathrm{min}}(\theta_2)$, which aligns with intuition: if possible, identify subsystems individually.  

\section{Numerical Example}
\label{sec:num}
We now compare the theoretical decay rate of topological misclassification derived in Theorem~\ref{thm:topo:equiv:rate} against the empirical decay rate for the nominal least squares estimator~$\widehat{\theta}_T$ and its reverse $I$-projection $\mathcal{P}(\widehat{\theta}_T)$. Concurrently, we exemplify the insights from Section~\ref{interconnections}. 
To this end, we set
\begin{equation*}
    Y = \begin{pmatrix} -0.1 & 1\\ 0.1 & 0.05\end{pmatrix}, 
\end{equation*}
and simulate system~\eqref{eq:LTI:system} for both $\theta_1$ and $\theta_2$ defined as in Section~\ref{interconnections} starting from $E=10^3$ initial conditions $x_0\overset{i.i.d.}{\sim}\mathcal{N}(0,I_4)$ with $S_w=I_4$. Each initial condition leads to a trajectory under both $\theta_1$ and $\theta_2$ from which we construct the corresponding least squares estimators $\widehat{\theta}^{(i)}_{j,T}$, $i=1,\dots,E$, $T=1,\dots,10^3$, $j\in \{1,2\}$. Averaging over the $E$ simulation runs yields the empirical probability that~$\widehat{\theta}_{j,T}$ or its reverse $I$-projection are topologically equivalent to the true system matrix $\theta_j$. 
\begin{figure}
    \centering
    \includegraphics[scale=0.2]{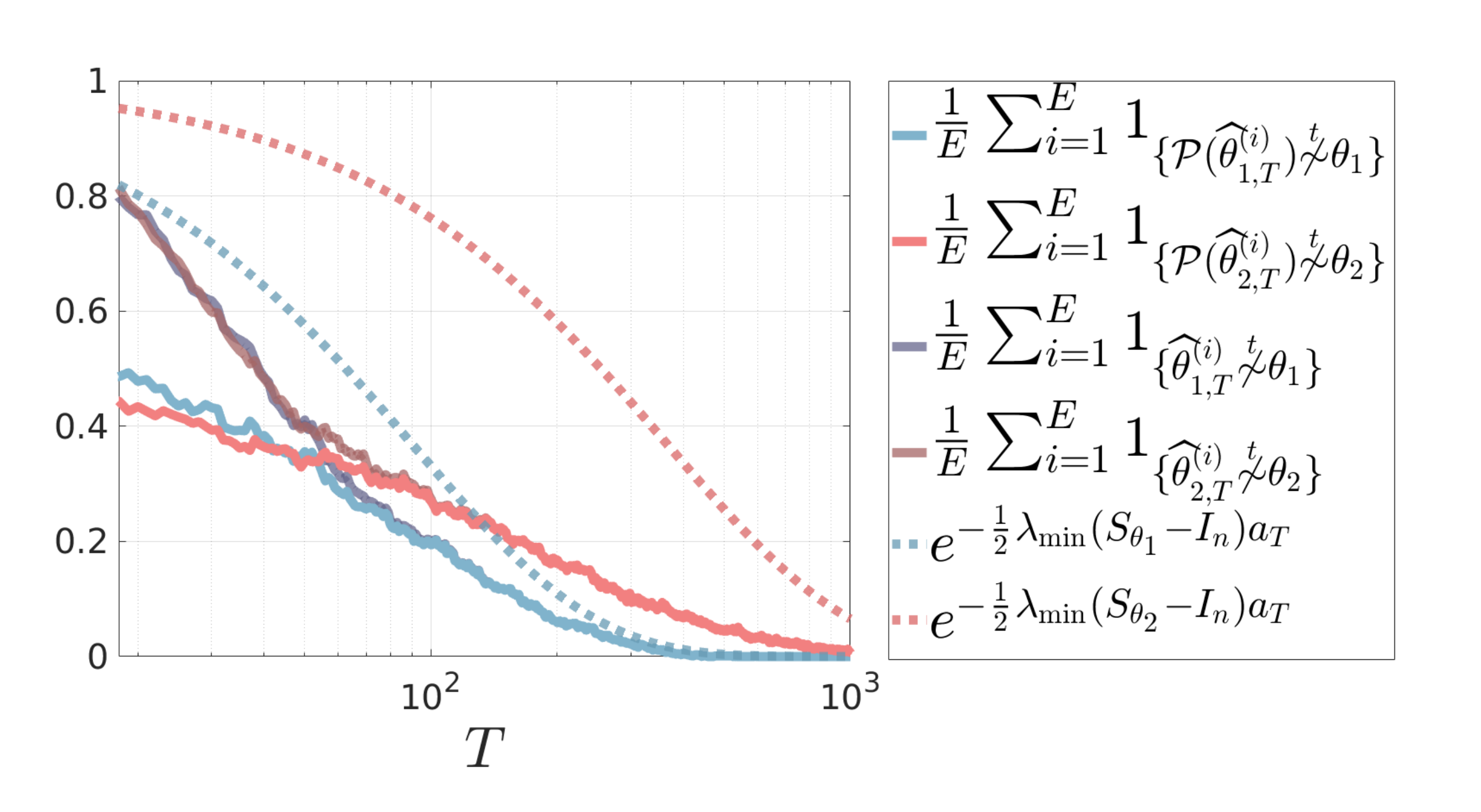}
    \caption{Empirical versus theoretical convergence rates.}
    \label{fig:num}
\end{figure}
Figure~\ref{fig:num} compares the bounds on the misclassification probability derived in Theorem~\ref{thm:topo:equiv:rate} for $a_T= T^{\frac{1}{1+\epsilon}}$ with $\epsilon=10^{-9}$ against the empirical probabilities. Here, $\mathcal{P}(\widehat{\theta}_{j,T}^{(i)})$ is computed via~\eqref{equ:rev;I:proj:comp} for $\delta=10^{-9}$. As expected, the projection accelerates topological identification, and block-diagonal system matrices are easier to identify. 

\printbibliography[title={Bibliography}]

\end{document}